\documentclass[oneside]{amsart}
\usepackage{amssymb}
\usepackage[foot]{amsaddr}
\usepackage{stmaryrd} 
\usepackage{amsmath} 
\usepackage{amscd}
\usepackage{cancel}
\usepackage[toc]{appendix}
\usepackage{amsthm}
\usepackage{amsbsy}
\usepackage{mathtools}
\usepackage{tablefootnote}
\usepackage{hyperref}
\usepackage{multirow}
\usepackage[normalem]{ulem}
\useunder{\uline}{\ul}{}

\usepackage{commath, longtable}
\usepackage{comment, enumerate}
\usepackage[marginparwidth=1in]{geometry}
\usepackage[matrix,arrow]{xy}
\usepackage{mathrsfs}
\usepackage{color}
\usepackage{float}
\usepackage{mathtools,caption}
\usepackage[table,dvipsnames,svgnames]{xcolor}
\usepackage{tikz-cd}
\usetikzlibrary{positioning}
\usepackage{longtable}
\usepackage{scalerel}
\usepackage[utf8]{inputenc}
\usepackage[OT2,T1]{fontenc}
\usepackage[normalem]{ulem}
\usepackage[customcolors]{hf-tikz}
\usepackage{enumitem}
\newlist{primenumerate}{enumerate}{1}
\setlist[primenumerate,1]{label={\arabic*$'$}}
\hypersetup{
 colorlinks=true,
 linkcolor=DarkOrchid,
 filecolor=blue,
 citecolor=olive,
 urlcolor=orange,
 pdftitle={HK p-TSNC},
 }
\usepackage{booktabs}
\usepackage{todonotes}

\DeclareSymbolFont{cyrletters}{OT2}{wncyr}{m}{n}
\DeclareMathSymbol{\Sha}{\mathalpha}{cyrletters}{"58}

\newtheorem{theorem}{Theorem}[section]
\newtheorem{lemma}[theorem]{Lemma}

\newtheorem{proposition}[theorem]{Proposition}
\newtheorem{corollary}[theorem]{Corollary}
\newtheorem{definition}[theorem]{Definition}

\numberwithin{equation}{section}
\newtheorem{lthm}{Theorem}

\theoremstyle{remark}
\newtheorem{remark}[theorem]{Remark}

\newcommand{\rhobar}{\bar{\rho}}

\newcommand{\GL}{\operatorname{GL}}
\newcommand{\SL}{\operatorname{SL}}
\newcommand{\PSL}{\operatorname{PSL}}
\newcommand{\PGL}{\operatorname{PGL}}

\newcommand{\F}{\mathbb{F}}
\newcommand{\Q}{\mathbb{Q}}
\newcommand{\C}{\mathcal{C}}

\newcommand{\Frob}{\operatorname{Frob}}

\newcommand{\tr}{\operatorname{tr}}
\newcommand{\Inn}{\mathrm{Inn}}
\newcommand{\Aut}{\mathrm{Aut}}
\newcommand{\GQ}{G_{\Q}}
\newcommand{\dm}[1]{\begin{psmallmatrix}#1\end{psmallmatrix}}

\makeatletter
\newcommand{\mylabel}[2]{#2\def\@currentlabel{#2}\label{#1}}
\makeatother

\title[Distinguishing elliptic curves mod $p$]{Distinguishing elliptic curves modulo $p$ and identifying images of product representations}

\author[J.~Bennett]{Jessica Bennett}
\address[Bennett]{
University of California, Irvine\\
Department of Mathematics\\
340 Rowland Hall\\
Irvine, CA 92697-3875 USA\\
bennett3@uci.edu
}

\author[J.~Hatley]{Jeffrey Hatley}
\address[Hatley]{
Department of Mathematics\\
Union College\\
Bailey Hall 202\\
Schenectady, NY 12308 USA\\ 
hatleyj@union.edu}

\author[S.~Kononova]{Sasha Kononova}
\address[Kononova]{University of California, Irvine\\
Department of Mathematics\\
340 Rowland Hall\\
Irvine, CA 92697-3875 USA\\
skononov@uci.edu}

\author[V.~Macri]{Vincent Macri}
\address[Macri]{
Department of Mathematics and Statistics\\
University of Calgary\\
Mathematical Sciences 476
Calgary, AB T2N 1N4 Canada\\
vincent.macri@ucalgary.ca
}

\author[Z.~Porat]{Zachary Porat}
\address[Porat]{
Department of Mathematics and Statistics\\
Bucknell University\\
380 Olin Science Building\\
Lewisburg, PA 17837 USA\\
z.porat@bucknell.edu}

\author[N.~Praveen]{Naina Praveen}
\address[Praveen]{
Department of Mathematics\\
University College London\\
Gower Street, London WC1E 6BT, UK\\
naina.praveen.24@ucl.ac.uk
}


\begin{document}

\begin{abstract} Given two elliptic curves defined over $\Q$ and a rational prime $p$, we study the product of their residual Galois representations. Using Goursat's lemma, we explicitly enumerate and completely characterize all possible images of such product representations. We also define associated invariants to these image groups, which we call \textit{witness ratios}, and we explain their computational utility and their relationship to the well-known Sturm bound for testing congruences between modular forms.
\end{abstract}

\maketitle

\section{Introduction} 

Two elliptic curves $E_1$ and $E_2$ defined over $\Q$ are said to be \textit{congruent modulo} a rational prime $p$ if there is a $\mathrm{Gal}(\bar{\Q}/\Q)$-module isomorphism $E_1[p] \simeq E_2[p]$ between their $p$-torsion subgroups. Congruences between elliptic curves have played an important role in many areas of research, such as the proof of the Modularity Conjecture \cite{Wil95, TW95,breuil2001modularity} and in studying the Iwasawa Main Conjecture \cite{greenbergvatsal,epw}. It has also been shown that congruences between elliptic curves allow for the transfer of arithmetic information \cite{shekhar-parity, hatleyparity}.

Elliptic curves between which there exists an isogeny of degree prime-to-$p$ are always congruent. For small primes $p$, congruences between non-isogenous curves can be found in abundance (e.g. \cite{rubin-silverberg}), while they become harder to find as $p$ gets larger (see e.g. \cite{fisher1}). In fact, the Frey-Mazur conjecture posits that for $p \geq 19$, congruences should only exist between isogenous elliptic curves. This conjecture has been verified for every pair of elliptic curves of conductor at most $500\,000$ \cite{cremona-freitas}.

In practice, given a pair of elliptic curves, one can effectively determine whether they are congruent using the following two facts. First, if $E_1[p] \simeq E_2[p]$, then for every prime $\ell \neq p$ of good reduction for both curves,
we must have $a_\ell(E_1) \equiv a_\ell(E_2) \mod p$, where $a_\ell(E)$ denotes the usual trace of Frobenius. Second, thanks to the Sturm bound (\cite[Corollary 9.20]{stein-book}), it suffices to check $\ell$ such that 
\begin{equation}\label{eq:sturm-bound}
\ell \leq
\frac{1}{6}N \prod_{q \mid N} \left(1 + \frac{1}{q}\right),
\end{equation}
where $N$ is the least common multiple of the conductors of the curves, and $q$ ranges over all primes dividing $N$. In particular, this bound grows like $O(N)$. It is worth noting that, in practice, if $E_1$ and $E_2$ are \textit{not} congruent, our computational experiments suggests that one rarely needs to go very far before finding a prime which witnesses this fact, i.e. a prime $\ell$ for which $a_\ell(E_1) \not\equiv a_\ell(E_2) \mod p.$ 

The purpose of this paper is to study the problem of distinguishing elliptic curves modulo $p$ from a slightly different perspective. In the remainder of this introduction, we summarize our approach and give some sample results.

Suppose that $E_1$ and $E_2$ are elliptic curves defined over $\Q$, and fix an odd prime $p$. Write $N_1$ and $N_2$ for their conductors, and for $i \in \{1,2\}$ denote by 
\[
\rhobar_{E_i,p} \colon \mathrm{Gal}(\bar{\Q}/\Q) \to \mathrm{GL}_2(\F_p)
\]
their associated mod $p$ Galois representations (see Section \ref{sec:background} for more background). We assume that these representations are surjective (and hence irreducible). Then $E_1$ and $E_2$ are congruent mod $p$ precisely when $\rhobar_{E_1,p}  \simeq \rhobar_{E_2,p} $. We recall that two such representations are isomorphic provided there exists an $M \in \mathrm{GL}_2(\F_p)$ such that
\[
\rhobar_{E_1,p}(\sigma) = M \rhobar_{E_2,p}(\sigma)M^{-1} \quad \text{for all } \sigma \in \mathrm{Gal}(\bar{\Q}/\Q).
\]
The Frobenius elements $\Frob_\ell$ associated to the good primes $\ell \nmid pN_1 N_2$ are dense in $\mathrm{Gal}(\bar{\Q}/{\Q})$ and satisfy $a_\ell(E_i) \equiv \tr(\rhobar_{E_i,p} (\Frob_\ell)) \mod p$. It follows that if $\rhobar_{E_1,p}  \not\simeq \rhobar_{E_2,p} $, then there exist primes $\ell$ such that $\tr(\rhobar_{E_1,p} (\Frob_\ell)) \not\equiv \tr(\rhobar_{E_2,p} (\Frob_\ell)) \mod p$.

Our strategy is to study the product representation
\[
\rhobar = \rhobar_{E_1,p}  \times \rhobar_{E_2,p}  \colon \mathrm{Gal}(\bar{\Q}/\Q) \to \mathrm{GL}_2(\F_p) \times \mathrm{GL}_2(\F_p)
\]
defined by $\sigma \mapsto (\rhobar_{E_1,p} (\sigma),\rhobar_{E_2,p} (\sigma))$ for all $\sigma \in \mathrm{Gal}(\bar{\Q}/\Q).$ On the one hand, if $\rhobar_{E_1,p}  \not\simeq \rhobar_{E_2,p} $, then there must exist $(g,h) \in \mathrm{Im}(\rhobar)$ such that $\tr(g) \not\equiv \tr(h) \mod p$; that is, $\mathrm{Im}(\rhobar)$ must not be conjugate to the diagonal subgroup.

On the other hand, the possible subgroups which can actually occur as the image of $\rhobar$ are greatly restricted; see Section \ref{subsec:approach}. When $p=3$ there are four non-diagonal possibilities, and when $p \geq 5$, there are three non-diagonal possibilities. 

Consider a subgroup $H \subset \mathrm{GL}_2(\F_p) \times \mathrm{GL}_2(\F_p)$, and partition it into conjugacy classes $C_1, \ldots, C_r$. Call a conjugacy class \textit{trace-different} if it contains elements $(g,h)$ such that $\tr(g)\not\equiv \tr(h) \mod p$; since trace is conjugacy-invariant, this is a well-defined notion which may be detected by sampling any representative element of the conjugacy class.

Let $\mathcal{S}$ be the set of trace-different conjugacy classes in $H$, and define
\begin{equation}\label{eq:WH-defn}
W(H)=\frac{1}{|H|} \sum_{C \in \mathcal{S}} |C|.
\end{equation}
By the Chebotarev Density Theorem, $W(H)$ may be interpreted as the asymptotic density of primes $\ell$ for which $a_\ell(E_1) \not\equiv a_\ell(E_2) \mod p$. We thus obtain theorems like the following; for a more complete description of the subgroups occurring, see sections \ref{sec:general-p} and \ref{sec:main-results}.

\begin{lthm}\label{thm:intro-sample-thm}
Suppose $E_1$ and $E_2$ are elliptic curves defined over $\Q$. Suppose further that $\rhobar_{E_1,3}$ and $\rhobar_{E_2,3}$ are surjective. If $\rhobar_{E_1,3} \not\simeq \rhobar_{E_2,3}$, then the image $H$ of the product representation is conjugate to one of four explicit groups, which we call $\Gamma_\chi, H_\pm, H_Q,$ and $\Delta$. Furthermore, we compute
\[
W(\bullet)=\begin{cases} 
    \frac{1}{4}, & \bullet = \Gamma_\chi \\ 
    \frac{5}{16}, & \bullet = H_\pm \\
    \frac{35}{64}, & \bullet = H_Q \\
    \frac{41}{64}, & \bullet=\Delta \\
    \end{cases}.
\]
As a consequence, we have
\[
\lim_{X \to \infty} \frac{\#\{\textrm{good primes } \ell < X \ |  \ a_\ell(E_1) \not\equiv a_\ell(E_2) \mod 3 \}}{\#\{\textrm{good primes } \ell < X\}} = W(H).
\]
\end{lthm}
This is proven as Theorem~\ref{thm:main-p-3}.

\begin{remark} In practice, this allows one to determine the image of the product representation for a concrete pair of elliptic curves. As an example, consider the elliptic curves with Cremona labels $E_1=\href{https://www.lmfdb.org/EllipticCurve/Q/11/a/2}{\texttt{11a1}}$ and $E_2=\href{https://www.lmfdb.org/EllipticCurve/Q/52/a/2}{\texttt{52a1}}$. Using the scripts found in \cite{code}, one finds that among the first $100\,000$ good primes, $64\,073$ of them witness the noncongruence of these curves modulo $3$. We have
\[
\frac{64\,073}{100\,000} \approx 0.64073.
\]
Comparing that with
\[
W(\Delta)=\frac{41}{64}=0.640625,
\]
one may reasonably guess that the image of the product representation is conjugate to $\Delta.$ In fact, this can be proven; see Section~\ref{subsec:realizations}.
\end{remark}

From Theorem~\ref{thm:intro-sample-thm}, one sees that in the worst case, a set of primes of density only $\frac{1}{4}$ witnesses the noncongruence modulo $3$. Consider the logarithmic integral function
\[
\mathrm{Li}(x)=\int_2^x \frac{dt}{\ln t}.
\]
Then an effective version of the Chebotarev Density Theorem (see e.g. \cite{DKN}) yields the following. 


\begin{corollary}
Retain the hypotheses of Theorem~\ref{thm:intro-sample-thm}. Then the number of noncongruence witnesses less than $X$ is asymptotic to
\begin{equation}\label{eq:chebotarev-1-4}
\left(\frac{1}{4}+o(1)\right)\mathrm{Li}(X).
\end{equation}
\end{corollary}

\begin{remark} We make the following remarks:
\begin{enumerate}
\item The first value of $\mathrm{Li}(x)$ which is greater than $4$ is $\mathrm{Li}(8)\approx 5.25$, so as a very rough heuristic, assuming that $p>7$ and that $2,3, 5$ and $7$ are primes of good reduction for both $E_1$ and $E_2,$ one might (very naively and optimistically) hope to witness noncongruence within the first four primes. 

\item Many explicit bounds on the error in the asymptotic~\eqref{eq:chebotarev-1-4} exist in the literature. Let $K/\Q$ be the degree $|H|$ extension cut out by the product representation $\rhobar$. Under the generalized Riemann hypothesis, Greni\'{e} and Molteni \cite{gren-molt} show that the error term is bounded above by
\[
\sqrt{X} \left[ \left( \frac{1}{2\pi}+\frac{3}{\ln(x)} \right) \ln(d_K) + \left( \frac{\ln(x)}{8\pi} + \frac{1}{4 \pi} + \frac{6}{\ln(x)} \right)|H| \right],
\]
where $d_K$ is the absolute discriminant of $K$. Note that $d_K$ can be bounded using the data of the conductors of the elliptic curves.

\item In a similar spirit, Bach and Sorenson \cite{bach-sorenson} prove (again under GRH) that \textit{every} conjugacy class (and therefore a noncongruence witness) should occur for 
\[
X \leq (1 + o(1))\left(\ln(d_K)\right)^2.
\] 

\item In principle, it should be possible to determine the size of the image representation by studying the intersection of the $p$-division fields $\Q(E_1[p])\cap\Q(E_2[p])$; see, for instance, \cite[Section 2]{daniels-hatley-ricci}. However, this approach quickly becomes computationally infeasible, since if $\rhobar_{E,p}$ is surjective, then 
\[ 
[\Q(E[p]):\Q]=|\mathrm{GL}_2(\F_p)|=p(p-1)^2(p+1).
\] 
When $p=3$, we have $|\mathrm{GL}_2(\F_p)|=48$, and \href{https://www.lmfdb.org/NumberField/?degree=48}{no fields of degree 48} are precomputed in the LMFDB \cite{lmfdb}. On the other hand, the results of this paper characterize the size of these intersections via very explicit and easily-computable criteria.
\item The comparison of residual Galois representations via ``witness'' Frobenius classes is also a fundamental technique in \cite{paramodularity}.
\end{enumerate}
\end{remark}

While the witness ratios described in Theorem~\ref{thm:intro-sample-thm} possess computational utility, we are also able to gain a theoretical understanding of when each admissible group occurs. When $p \geq 5$, we obtain the following result (see Section~\ref{subsec:realizations}), where we write $\chi_d$ for the quadratic character of conductor $d$ and $p^\ast=(-1)^{\frac{p-1}{2}}p$.

\begin{lthm}
Let $E_1,E_2$ be elliptic curves over $\Q$ with surjective mod $p$ Galois representations $\rhobar_{E_i,p}$, where $p \geq 5$. Let $H$ be the image of the product representation $\rhobar_{E_1,p} \times \rhobar_{E_2,p}$. Then $H$ is conjugate to:
\begin{enumerate}
\item $\Gamma_\mathrm{id}$ if and only if $\rhobar_{E_1,p} \simeq \rhobar_{E_2,p}$. 
\item $H_\pm$ if and only if $\rhobar_{E_2,p} \simeq \chi_{d} \otimes \rhobar_{E_1,p}$ for some squarefree $d \notin\{1,p^\ast\}$. Thus pairs of quadratic twists $(E,E^{(d)})$ give examples of $H_\pm$.
\item $\Gamma_\chi$ if and only if $\rhobar_{E_2,p} \simeq \chi_{p^\ast} \otimes \rhobar_{E_1,p}$.  Thus pairs of quadratic twists $(E,E^{(p^\ast)})$ give examples of $\Gamma_\chi$.
\item $\Delta$ in every other case.
\end{enumerate}
\end{lthm}
This is proven as Theorem~\ref{thm:explicit-realizations}. See Theorem~\ref{thm:3-delta} for a description of the extra admissible group when $p=3.$

\begin{remark}It would be interesting to extend this study to the case of non-surjective representations arising from elliptic curves, as well as more general modular Galois representations taking values in $\GL_2(\F_{p^m})$.\end{remark}

\section*{Acknowledgements} The authors thank Allechar Serrano L\'{o}pez, Heidi Goodson and Jen Berg for bringing them together through the Rethinking Number Theory (RNT7) workshop in the summer of 2026. RNT7 is supported by NSF grant DMS-2418528. 

The authors are extremely grateful to Harris Daniels for helpful conversations related to this paper, and especially for providing a Magma implementation of Goursat's lemma.

Partial support for this research was provided by an AMS-Simons Research Enhancement Grant for Primarily Undergraduate Institution Faculty. Partial support for this research was provided by a generous gift to the Union College Mathematics Department by David '74 and Jeanette Wagner.

\section*{A note on AI use during the preparation of this paper} 
The genesis of this paper, the introduction, the background material, and the original results (e.g. Theorem~\ref{thm:intro-sample-thm}) were generated solely by human thought and computational experimentation. Later, as a further experiment, the authors queried Claude (running the Fable 5 model) about whether their results could be generalized in an explicit fashion for general $p$. Claude proposed the general statements and a proof outline, which the authors checked, corrected, and filled out, resulting in Section~\ref{sec:general-p}. At various points, the human authors had to make corrections to false assertions made by the AI. Ultimately, we take full responsibility for the soundness of the arguments therein, and we make this disclosure in the spirit of the \href{https://leidendeclaration.ai}{Leiden Declaration}.

\section{Background on elliptic curves and Galois representations}\label{sec:background}

In this section we very briefly recollect some of the most important facts about the Galois representations associated to elliptic curves. To that end, let $E/\Q$ be an elliptic curve of conductor $N$, and let $p$ be a rational prime. By considering the action of $G_\Q=\mathrm{Gal}(\bar{\Q}/\Q)$ on the $p$-torsion points $E[p]$, one obtains a 2-dimensional Galois representation which we will denote $\rhobar_{E,p}$. More precisely, after choosing a basis for $E[p]$, we obtain a continuous homomorphism
\[
\rhobar_{E,p} \colon G_\Q \to \mathrm{GL}_2(\F_p).    
\]
Since the image of this map is finite, this representation cuts out a finite extension $K/\Q$ of the rationals, where $K=\bar{\Q}^{\mathrm{ker}\bar{\rho}_{E,p}}$ is the fixed field of the kernel. More explicitly, we have $K=\Q(E[p])$ is the $p$-division field, obtained by adjoining to $\Q$ the coordinates of all $p$-torsion points. 

If we write $G=\mathrm{Gal}(K/\Q)$, then we have $G \simeq \mathrm{Im}(\bar{\rho}_{E,p}) \subset \mathrm{GL}_2(\F_p).$ By a result of Serre, this map is actually surjective for most $p$, provided $E$ does not have complex multiplication; in general, a classification due to Dickson lists the possible subgroups of $\mathrm{GL}_2(\F_p)$ which can arise in this way \cite{serre72}. Throughout this paper, we will assume that our residual representations $\bar{\rho}_E$ are surjective.

Since $K/\Q$ is a number field, one can ask how rational primes decompose in this extension. The primes that ramify are contained in those that divide $Np$. If $\ell \nmid Np$, then $\ell$ is unramified in $K/\Q$, and so the conjugacy class $\mathrm{Frob}_\ell$ is well-defined. For such $\ell$, we have: 

\begin{enumerate}
\item $\tr(\bar{\rho}_E(\mathrm{Frob_\ell)}) \equiv a_\ell(E) \mod p$; \quad \text{and} 
\item $\mathrm{det}(\bar{\rho}_E(\mathrm{Frob_\ell)}) \equiv \ell \mod p$.
\end{enumerate}

(In the first bullet, we use the standard notation for the invariant $a_\ell(E)=\ell+1-\#\tilde{E}(\F_\ell)$. The second bullet is a specialization of the more general fact that $\det \bar{\rho}_E$ is the mod $p$ cyclotomic character, which we will denote by $\chi_{cyc}$.)

Recall that $G=\mathrm{Gal}(K/\Q)$ and fix a conjugacy class $C$ inside $G$. The Chebotarev Density Theorem tell us that the rational primes $\ell$ for which $\mathrm{Frob}_\ell$ belongs to $C$ has density $\frac{|C|}{|G|}$. One consequence of this theorem is that the Frobenius elements at good $\ell$ (which excludes only finitely many) are dense in $G_\Q$. It then follows from the Brauer-Nesbitt Theorem that properties (1) and (2) above completely determine $\bar{\rho}_E$ up to semisimplification. 

\subsection{The motivating question} Now suppose we have \textit{two} elliptic curves, $E_1/\Q$ and $E_2/\Q$, with conductors $N_1$ and $N_2$, respectively. Then one may wish to compare $\rhobar_{E_1,p}$ and $\rhobar_{E_2,p}$; for instance, one may wish to determine whether $\rhobar_{E_1,p} \simeq \rhobar_{E_2,p}$. (Recall that two representations are said to be isomorphic if there exists a matrix $M \in \mathrm{GL}_2(\F_p)$ such that $\rhobar_{E_2,p} (\sigma)=M\rhobar_{E_1,p}(\sigma)M^{-1}$ for all $\sigma \in G_\Q$.)

As mentioned in the introduction, given a pair of elliptic curves, one can effectively determine whether they are congruent using the following two facts. First, as explained above, if $\rhobar_{E_1,p} \simeq \rhobar_{E_2,p}$, then for every $\ell \nmid pN_1 N_2$, we must have $a_\ell(E_1) \equiv a_\ell(E_2) \mod p$. Second, it suffices to check only finitely many such $\ell$ values, thanks to the Sturm bound (\cite[Corollary 9.20]{stein-book}), which was stated in \eqref{eq:sturm-bound}.

This project was originally motivated by a desire to study the sharpness of the Sturm bound and to seek more efficient heuristics for distinguishing non-isomorphic Galois representations.

\subsection{Our approach}\label{subsec:approach} Retain the notation of the previous section. Recall our assumption that the representations $\rhobar_{E_i,p}$ are surjective. 

In addition to the individual representations $\rhobar_{E_i,p}$, we can define the product representation
\[
\rhobar = \bar{\rho}_{E_1,p} \times \bar{\rho}_{E_2,p} ~\colon G_\Q \to \mathrm{GL}_2(\F_p) \times \mathrm{GL}_2(\F_p)
\]
via the diagonal embedding
\[
\sigma \mapsto (\bar{\rho}_{E_1,p}(\sigma), \bar{\rho}_{E_2,p}(\sigma)) \quad \text{for all } \sigma \in G_\Q.
\]
Let us write $H=\mathrm{Im}(\rhobar)$ for the image of this representation. This is a subset of $\mathrm{GL}_2(\F_p) \times \mathrm{GL}_2(\F_p)=G_1 \times G_2$, and since each $\rhobar_{E_i,p}$ is surjective, the projections $\pi_1 \colon H \to G_1$ and $\pi_2 \colon H \to G_2$ are both surjective. Furthermore, since $\det \rhobar_{E_i,p}=\chi_{cyc}$ for both $i=1,2$ we must have $\det(g_1)=\det(g_2)$ for all $(g_1,g_2) \in H.$

\begin{definition}\label{def:admissible}
Let $G=\GL_2(\F_p)$ and let $H$ be a subgroup of $G \times G$. We say that $H$ is \textit{admissible} if $H$ surjects onto each factor of $G$, and for all $(g_1,g_2) \in H$ we have $\det(g_1)=\det(g_2)$. 
\end{definition}

 Goursat's lemma clarifies the possibilities for $H$.

\begin{lemma}[Goursat]\label{lem:goursat}
Let $A,B$ be groups and $H\leq A\times B$ a subgroup projecting surjectively onto each factor. Set $N_A=\{a : (a,1)\in H\}\trianglelefteq A$ and $N_B=\{b:(1,b)\in H\}\trianglelefteq B$. Then $H$ induces an isomorphism $\varphi\colon A/N_A\xrightarrow{\;\sim\;}B/N_B$ and
\[
H=\{(a,b) : \varphi(aN_A)=bN_B\}.
\]
In particular $N_A\times N_B\subseteq H$.
\end{lemma}

In particular, the possible subgroups $H$ which can arise as the image of $\rhobar$ are, up to conjugacy, parameterized by pairs $(N,\phi)$, where $N \triangleleft \mathrm{GL}_2(\F_p)$ is a normal subgroup and $\phi \in \mathrm{Aut}(G/N)$.

Given the scarcity of normal subgroups of $\mathrm{GL}_2(\F_p)$, the number of 
admissible image groups $H$ is extremely limited. For example, when $p=3$ there are only four non-diagonal admissible image groups, and when $p>3$ there are only three. We describe the situation for $p \geq 5$ in Section~\ref{sec:general-p} and the special case of $p=3$ in Section~\ref{sec:main-results}.

If $\rhobar_{E_1,p} \simeq \rhobar_{E_2,p}$, then $\mathrm{Im}(\rhobar)=H$ will be the diagonal subgroup; otherwise, we will have some rational prime $\ell$ such that 
\[
(\bar{\rho}_{E_1,p} \times \bar{\rho}_{E_2,p})(\mathrm{Frob_\ell)} = (g_1,g_2) \quad \text{with} \quad \tr(g_1) \neq \tr(g_2).
\]
Call such a pair $(g_1,g_2)$ \textit{trace-different.} Since trace is invariant under conjugation, and since conjugacy classes correspond to Frobenius elements, the number of such trace-different elements in $H$ gives rich arithmetic information about the elliptic curves, their Galois representations, and the distribution of their Frobenius traces, as we show. This motivates the following definition.
\begin{definition}\label{def:witness-ratio}
For an admissible subgroup $H \subset \mathrm{GL}_2(\F_p) \times \mathrm{GL}_2(\F_p)$, define its \textit{witness ratio} by
\[
W(H)=\frac{\#\{(g_1,g_2) \in H \ | \ \tr(g_1) \neq \tr(g_2)\}}{|H|}.
\]
Equivalently, it is defined by Equation~\eqref{eq:WH-defn}.
\end{definition}
The main goal of this paper is to classify the admissible image subgroups for each prime $p$ and to compute their witness ratios. From this information, we can make a few deductions. For example:

\begin{enumerate}
\item Consider the value $\delta(p) = \min_{\text{admissible}\ H} W(H)$. Then if $\rho_{E_1,p}$ and $\rho_{E_2,p}$ are distinct surjective mod $p$ representations coming from elliptic curves, then their Frobenius traces must differ on a set of primes of \textit{density} at least $\delta(p)$.

\item For a concrete pair of curves, one may compute the ratio
\[
 \frac{\#\{\textrm{good primes } \ell < X \ |  \ a_\ell(E_1) \not\equiv a_\ell(E_2) \mod p \}}{\#\{\textrm{good primes } \ell < X\}}
\]
for large values of $X$, and in practice, this value will approach $W(H_i)$ for some admissible group $H_i$. One may then reasonably conclude that the image of the product representation is conjugate to $H_i$. See Theorem~\ref{thm:explicit-realizations} for an explicit description of a non-probabilistic determination. See the file \texttt{extended\_ncr.sage} in \cite{code} for a SageMath implementation of computing this ratio.
\end{enumerate}

\subsection{Relation to entanglements of division fields} For the elliptic curves $E_1$ and $E_2$ and our fixed prime $p$, consider the corresponding $p$-division fields $K_i=\Q(E_i[p])$ obtained by adjoining to $\Q$ all of the coordinates of the points in $E_i(\bar{\Q})[p]$. Then $\operatorname{Im}(\rhobar_{E_i,p}) \simeq \operatorname{Gal}(K_i/\Q)$ for each $i=1,2$. Furthermore, the existence of the Weil pairing forces $\Q(\zeta_p) \subset K_i$, where $\zeta_p$ is a primitive $p$-th root of unity. Thus $K_1 \cap K_2 \supset \Q(\zeta_p).$ 

The image of the product representation $\rhobar$ is isomorphic to the Galois group of the compositum $K_1.K_2 / \Q$. If $K_1 \cap K_2 = \Q(\zeta_p)$, then this product representation is as large as possible, and this should be the generic case. (Our computations, explained in Section~\ref{sec:main-results}, support this.) On the other hand, it can sometimes happen that $K_1$ and $K_2$ have a larger intersection, resulting in a smaller $\operatorname{Im}(\rhobar)$. Such \textit{entanglements} have been the subject of much recent study \cite{DLM22, DLM23, DLRM, DGL}, though most of the literature studies the entanglement between $p$- and $q$-division fields for the same elliptic curve. In any case, since the degrees involved make explicitly computing division fields prohibitively expensive for most primes $p$, our methods yield an alternative method to effectively determine the entanglement of two elliptic curves.

\section{General results for $p \geq 5$}\label{sec:general-p}

In this section we classify the possible subgroups for $\rhobar$ for all $p\geq 5$, and we compute their witness ratios. We also characterize the pairs of elliptic curves that realize each subgroup, and we illustrate this with some concrete examples and explicit computations. To that end, let us set some notation for the rest of the section. We fix a prime $p \geq 5$ and set $G=\GL_2(\F_p)$ and $S=\SL_2(\F_p)$. Recall that 
\[
|G|=p(p-1)^2(p+1).
\] 
Denote by $Z\cong\F_p^\times$ the center of $G$. Write $\chi$ for the quadratic character of $\F_p^\times$ (i.e. the Legendre symbol mod $p$).

Define $J=\dm{1&1\\0&1}$, $J_\epsilon=\dm{1&\epsilon\\0&1}$ with $\epsilon \in \F_p^\times$ a fixed nonsquare. Set
\[
\Delta=\{(g,h)\in G\times G:\det g=\det h\},\qquad \text{so} \quad |\Delta|=|G|^2/(p-1).
\]
Suppose we have fixed elliptic curves $E_1,E_2/\Q$ with surjective residual representations $\rhobar_{E_i,p}\colon\GQ\to G$; since $\det\rhobar_i=\bar\chi_{\mathrm{cyc}}$ is the cyclotomic character, the image $H$ of $\rhobar=\rhobar_{E_1,p}\times\rhobar_{E_2,p}$ lies in $\Delta$ and projects onto both factors. We work up to conjugacy in $G\times G$, so that our arguments are independent of our choice of basis of $E_1[p]$ and $E_2[p]$. 

The results in this section were discovered with the help of Claude (Fable 5). 

\subsection{Classifying the admissible subgroups}\label{subsec:classifying-admissible-subgroups} Let us define the following subgroups of $G \times G:$
\[
\Gamma_{\mathrm{id}}=\{(g,g) \ | \ g \in G\},\quad
\Gamma_{\chi}=\{(g,\chi(\det g)g)\ | \ g \in G\},\quad \text{and} \quad
H_{\pm}=\{(g,\pm g) \ | \ g \in G\},
\]
with $\Gamma_{\mathrm{id}},\Gamma_\chi\subset H_\pm\subset\Delta$, the first two of index $2$ in the third. We see that $\Gamma_{\mathrm{id}}$ is the usual {\it diagonal subgroup}, $\Gamma_\chi$ is the $\chi${\it -twisted diagonal} subgroup, $H_\pm$ is the {\it sign-twisted diagonal} subgroup, and $\Delta$ is the full \textit{determinant-equal} subgroup. In this section, we will prove the following theorem.

\begin{theorem}\label{thm:class} Let $p \geq 5$. Every $H\leq\Delta$ with surjective projections is $(G\times G)$-conjugate to exactly one of $\Gamma_{\mathrm{id}},\Gamma_\chi,H_\pm,\Delta$.
\end{theorem}

\subsubsection{Normal subgroups of $G$}\label{subsubsec:normal-subgroups-of-G}
First we classify the normal subgroups of $G$ and their automorphism groups. We begin with the following useful lemma.

\begin{lemma}\label{lem:center-of-S}
The centralizer of $S$ in $G$ is $Z$. More generally, if $g \in G$ satisfies $gsg^{-1}s^{-1} \in \{\pm 1\}$ for all $s \in S$, then $g \in Z$.
\end{lemma}
\begin{proof} Suppose $g \in G$ satisfies $gsg^{-1}s^{-1} \in \{\pm 1\}$ for all $s \in S$. Then the map $\phi(s) = gsg^{-1}s^{-1}$ is a homomorphism $\phi \colon S \to \{\pm 1\}$. It is a standard fact that $\SL_2(\F_p)$ is perfect (i.e. its own commutator subgroup) for $p \geq 5$, so this map must in fact be trivial. Thus we have $gsg^{-1}=s$ for all $s\in S$, so $g$ is in the centralizer of $S$.

Write $g=\left(\begin{smallmatrix} a & b \\ c & d \end{smallmatrix} \right)$; then since $g$ is in the centralizer of $S$, we have in particular $gJ=Jg$, which immediately implies $c=0$ and $a=d$. Computing $\left(\begin{smallmatrix}a & b \\ 0 & a\end{smallmatrix}\right)J^T=J^T\left(\begin{smallmatrix}a & b \\ 0 & a\end{smallmatrix}\right)$ now implies that also $b=0$. So we have $g=aI$, and since $g \in G$, we must have $a \in \F_p^\times$, hence $g \in Z$ as desired.
\end{proof}

\begin{proposition}\label{prop:normal-subs-are-in-center-or-big}
If $N \triangleleft G$ is a normal subgroup, then $N \subseteq Z$ or $S \subseteq N$.
\end{proposition}
\begin{proof}
Recall that $\PSL_2(\F_p)$ is simple for $p \geq 5$. Let $\tilde{N}$ denote the image of $N$ inside $\PGL_2(\F_p)$; then $\tilde{N} \cap \PSL_2(\F_p)$ is either trivial or $\PSL_2(\F_p)$. 

First suppose the intersection is trivial. Then the commutator subgroup $[\tilde{N},\PSL_2(\F_p)] \subset \tilde{N} \cap \PSL_2(\F_p)$ is trivial, so $\tilde{N}$ is contained in the centralizer $C$ of $\PSL_2(\F_p)$ in $\PGL_2(\F_p)$. If we show that $C$ is trivial, then it will follow that $N \subseteq Z.$ Let $\tilde{g} \in C$ and consider any lift $g \in G$. Then for any $s \in S$ we must have 
\begin{equation}\label{eq:psl-centralizer}
gsg^{-1}=\left(\begin{smallmatrix} \lambda(s) & 0 \\ 0 & \lambda(s) \end{smallmatrix} \right)s \quad  \text{for some} \ \lambda(s)\in \F_p^\times,    
\end{equation} and this defines a homomorphism $\phi \colon S \to \F_p^\times.$ Composing with the determinant, we see from \eqref{eq:psl-centralizer} that $\lambda^2(s)=1$, hence $\mathrm{Im}(\phi) \subset \{\pm1\}$. It now follows from Lemma~\ref{lem:center-of-S} that $g \in Z$, hence $\tilde{g}=1$ as desired.

In the second case, we have $\PSL_2(\F_p) \subseteq \tilde{N}$, hence $SZ \subseteq NZ$. But since $S$ is its own commutator, we have
\[
S=[S,S]=[SZ,SZ] \subseteq [NZ,NZ] = [N,N] \subseteq N.
\]
So $N$ contains $S$ as desired.
\end{proof}

\begin{remark}\label{rem:prop-fails-p3}
We note that $\PSL_2(\F_3) \simeq A_4$ is not simple, and the argument above fails for $p=3$; indeed, $\GL_2(\F_3)$ contains an additional normal subgroup, normally denoted $Q_8$, which neither contains $\SL_2(\F_3)$ nor is contained in the scalar diagonal matrices. However, our explicit computations, described later in Section~\ref{sec:main-results} and Theorem~\ref{thm:main-p-3}, along with Theorem~\ref{thm:3-delta}, handle this special case. 
\end{remark}

\subsubsection{Fiber products restricted by the determinant condition}\label{subsubsec:fiber-products-over-determinant}

Recall that, by Lemma~\ref{lem:goursat}, if $H \subseteq G \times G$ is a subgroup which surjects onto each factor, then there exist normal subgroups $N_1, N_2 \triangleleft G$ such that $N_1 \times N_2 \subseteq H$ and there is an isomorphism $\phi \colon G/N_1 \to G/N_2$ whose graph is $H$. In particular, we must have $|N_1|=|N_2|$. 

Furthermore, by Proposition~\ref{prop:normal-subs-are-in-center-or-big}, either $S \subseteq N_i$ or $N_i \subseteq Z$. Since $S=[G,G]$, if $S \subseteq N_i$ then $G/N_i$ is abelian, while if $N \subseteq Z$, then $G/N \twoheadrightarrow G/Z=\PGL_2(\F_p)$ which is nonabelian. So either $N_1$ and $N_2$ are both central, or they both contain $S$. 

In our setting, $H=\mathrm{Im}(\rhobar)$ is the image of a product representation $\rhobar_{E_1,p}\times\rhobar_{E_2,p}$; since $\det(\rhobar_{E_i,p})=\chi_{cyc}$ for both $i=1,2$, if $(g,h)\in H$ then $\det(g)=\det(h)$.

\begin{lemma}\label{lem:possibility-delta}
Suppose $N_1$ and $N_2$ contain $S$. Then $H=\Delta$ is the full determinant-equal subgroup.
\end{lemma}
\begin{proof}
Since $N_1 \times N_2 \subseteq H$, our hypothesis implies $S \times S \subseteq H$. Then $H/(S \times S)$ is a subgroup of $\F_p^\times \times \F_p^\times$, and the determinant condition forces it to be contained in the diagonal of $\F_p^\times \times \F_p^\times$, hence its preimage is contained in $\Delta$. On the other hand, since the projection to each factor is surjective, this containment must in fact be an equality, so $H=\Delta$ as claimed.
\end{proof}


\begin{lemma}\label{lem:N-size-1-or-2}
Suppose $N_1$ and $N_2$ contain $Z$. Then either $N_1=N_2=\{I\}$ or $N_1=N_2=\{\pm I\}$. 
\end{lemma}
\begin{proof}
Suppose $N_1$ and $N_2$ contain $Z \simeq \F_p^\times$; in particular, for both $i=1$ and $i=2,$ every element of $N_i$ is of the form $\lambda I$ for some $\lambda \in \F_p^\times$. By Goursat's lemma, if $\lambda I \in N_1$, then $(\lambda I, I) \in H$. But then the determinant condition forces $\lambda^2=1$. Since $|N_1|=|N_2|$, the result follows.
\end{proof}

By Goursat's lemma, automorphisms of normal subgroups $N \triangleleft G$ can lead to different fiber products. The next lemma classifies the automorphisms of $G$ that preserve determinants. Recall that $\chi$ denotes the quadratic character of $\F_p^\times.$
 
\begin{lemma}\label{lem:det-preserving}
Suppose $\alpha \in \Aut(G)$ satisfies $\det \circ \alpha = \det$. Then $\alpha \in \Inn(G)$ or $\alpha \in \Inn(G)\cdot r_\chi$, where $r_\chi(g)=\chi(\det g)g$. Furthermore, $r_\chi \not\in \Inn(G).$
\end{lemma}

\begin{proof}
Since $S=[G,G]$ and commutators are characteristic subgroups, any automorphism of $G$ restricts to an automorphism of $S$. It is well-known that for $p \geq 5$, $\Aut(S) \simeq \PGL_2(\F_p)$ acting via conjugation, so after composing with some inner automorphism, $\alpha \in \Aut(G)$ can be assumed to fix $S$ pointwise. Thus for $g \in G$ and $s \in S$ we have
\[
\alpha(g)s\alpha(g)^{-1}=\alpha(gsg^{-1})=gsg^{-1},
\]
which implies $g^{-1}\alpha(g)$ is in the centralizer of $S$, hence in $Z$ by Lemma~\ref{lem:center-of-S}. Thus $\alpha(g)=\lambda(g)g$ for some map $\lambda \colon G \to \F_p^\times$; in fact, it must be a homomorphism since 
\[
\lambda(g_1 g_2)g_1 g_2=\alpha(g_1 g_2) = \alpha(g_1) \alpha(g_2) = \lambda(g_1)\lambda(g_2)g_1 g_2.
\]

Since $\alpha$ fixes $S$ pointwise, we see that $\lambda|_S$ is trivial, so we can decompose $\lambda$ as $\lambda = \tilde{\lambda} \circ \det$ for some character $\tilde{\lambda}$ on $\F_p^\times$. On the other hand, we are assuming $\det \circ \alpha = \det$, so a direct computation shows that we must have $\tilde{\lambda}^2=1$, hence $\tilde{\lambda} \in \{1, \chi\}$. (Since $\F_p^\times$ is cyclic of even order, $\chi$ is its \textit{unique} quadratic character.) Thus $\lambda \in \{1, \chi \circ \det\}$, so either $\alpha \in \Inn(G)$ or $\alpha \in \Inn(G) \cdot r_\chi$ as claimed. 

The only claim left to check is that $r_\chi \not\in \Inn(G)$. Suppose $r_\chi \in \Inn(G)$; then there would exist some matrix $c \in G$ such that
\[
cgc^{-1}=r_\chi(g),
\]
and in particular, for every $g$ with nonsquare determinant, we must have
\[
cgc^{-1}=-g.
\]
Comparing traces, this would imply that every such $g$ has $\tr(g)=0$, but this is clearly false (e.g. take $g=\left(\begin{smallmatrix}\epsilon & 0 \\ 0 & 1\end{smallmatrix}\right)$ where $\epsilon \not\in (\F_p^\times)^2$ and $\epsilon \neq -1$).
\end{proof}

\subsubsection{Proof of Theorem~\ref{thm:class}}\label{subsubsec:proof-of-classification-theorem} We are now ready to prove that every $H\leq\Delta$ with surjective projections is $(G\times G)$-conjugate to exactly one of $$\Gamma_{\mathrm{id}},\quad \Gamma_\chi, \quad H_\pm,\quad \text{and} \quad \Delta.$$ By Goursat's Lemma, $H$ is determined by a pair $(N_1,N_2)$ of normal subgroups of $G$ with $G/N_1 \simeq G/N_2$. In particular, $|N_1|=|N_2|$. Proposition~\ref{prop:normal-subs-are-in-center-or-big} shows that for $i=1,2$, either $S \subseteq N_i$ or $N_i \subseteq Z$. 

By Lemma~\ref{lem:possibility-delta}, if $S \subseteq N_i$, then $H=\Delta$, which gives one of the claimed possibilities.

Otherwise, we must have $N_i \subseteq Z$. By Lemma~\ref{lem:N-size-1-or-2}, we have either $N_1=N_2=\{I\}$ or $N_1=N_2=\{\pm I\}$. Let us consider each case in turn.

First, suppose $N_1=N_2=\{I\}$. Then by Goursat's lemma, we obtain a subgroup of the form
\[
H_\alpha=\{(g, \alpha(g)) \ | \ g \in G\} \subseteq \Delta, 
\]
where $\alpha \in \Aut(G)$ must be determinant-preserving. Recall that we only care about admissible subgroups up to conjugation, thus by Lemma~\ref{lem:det-preserving}, this means that either $\alpha \in \Inn(G)$ or $\alpha \in \Inn(G) \cdot r_\chi$, where $\chi$ is the quadratic character on $\F_p^\times$. If $\alpha \in \Inn(G)$, then $H_\alpha$ is conjugate to $\Gamma_\mathrm{id}$, and if $\alpha \in \Inn(G) \cdot r_\chi$, then $H_\alpha$ is conjugate to $\Gamma_\chi$.

Finally, suppose $N_1=N_2=\{ \pm I\}$. Let us consider the corresponding quotient $\bar{G}=G/\{\pm I\}$, writing $g \mapsto \bar{g}$ at the level of elements. Thus $G/N_i \simeq \bar{G}$, and by Goursat's lemma we have 
\[
H=\{(g_1,g_2) | \ g_1,g_2 \in G \ \text{ and }\ \bar{\alpha}(\bar{g_1})=\bar{g_2}\ \text{ for some } \bar{\alpha} \in \Aut(\bar{G})\}.
\]
Since $\det(-g)=\det(g)$ and $H \subseteq \Delta$, we must still have $\det \circ \bar{\alpha}=\det$. In contrast to Lemma~\ref{lem:det-preserving}, we will show that every such $\bar{\alpha}$ belongs to $\Inn(\bar{G})$. We have $[\bar{G},\bar{G}]=\PSL_2(\F_p)$ and $\Aut(\PSL_2(\F_p))\simeq \PGL_2(\F_p)$. Furthermore, by Lemma~\ref{lem:center-of-S}, we see that the centralizer of $\PSL_2(\F_p)$ in $\bar{G}$ is exactly $Z/\{\pm I\}$ (this is the content of the ``more general'' statement in that lemma). So by the exact same argument as in Lemma~\ref{lem:det-preserving}, we may assume that $\bar{\alpha}$ fixes $\PSL_2(\F_p)$ pointwise, from which it follows that $\bar{\alpha}(\bar{g})=(\tilde{\alpha} \circ \det)(\bar{g}) \bar{g}$ with $\tilde{\alpha} \colon \F_p^\times \to \F_p^\times / \{\pm 1\}$. But just as before, the preservation of the determinant forces $\tilde{\alpha}^2=1$, hence in this case $\tilde{\alpha}$ is in fact trivial. This shows that if $\alpha \in \Aut(\bar{G})$ preserves determinants, then $\alpha \in \Inn(\bar{G})$, so up to conjugation we obtain only one admissible subgroup in this case, namely
\begin{align*}
H_\pm &= \{ (g_1, g_2) \ |\ g_1, g_2 \in G \text{ such that } \ \bar{g_1}=\bar{g_2} \} \\
&=\{(g,\pm g) \ |\ g \in G\}. 
\end{align*}

The last thing to verify is that the four admissible groups we have listed, namely 
$$\Gamma_{\mathrm{id}},\quad \Gamma_\chi, \quad H_\pm, \quad \text{and} \quad \Delta,$$
are all distinct up to conjugacy. The only pair that is not distinguished by order is the set $\{\Gamma_\mathrm{id},\Gamma_\chi\}$. Suppose these groups were conjugate; then there exists some $(a,b)\in G \times G$ such that for all $g \in G$ we have 
\begin{align*}
bgb^{-1}&=\chi(\det(aga^{-1}))aga^{-1}, 
\end{align*}
and since $\det$ is conjugacy-invariant, this implies 
\[
\chi(\det(g))g = (a^{-1}b)g(a^{-1}b)^{-1}.
\]
But this implies $r_\chi$ is given by conjugation by $a^{-1}b$, contradicting Lemma~\ref{lem:det-preserving}. So this is impossible, hence $\Gamma_\mathrm{id}$ and $\Gamma_\chi$ are not conjugate.
\hfill $\square$

\subsection{Conjugacy classes of the admissible subgroups}\label{subsec:conjugacy-classes}

For each of the admissible subgroups catalogued in Section~\ref{subsec:classifying-admissible-subgroups}, we will now determine their conjugacy classes. First, let us recall the conjugacy classes of $G$:
\begin{itemize}
\item Let $\mathcal{D}_{a,b}$ be the set of diagonalizable matrices with eigenvalues $a, b \in \F_p^\times$ with $a \neq b$. There are $\frac{1}{2}(p-1)(p-2)$ choices\footnote{Note that $\left(\begin{smallmatrix}
    a & 0 \\ 0 & b
\end{smallmatrix}\right)$ and $\left(\begin{smallmatrix}
    b & 0 \\ 0 & a
\end{smallmatrix}\right)$ are conjugate by $\left(\begin{smallmatrix}
    0 & 1 \\ 1 & 0
\end{smallmatrix}\right)$.} 
of $(a,b)$, and for each choice, we have $|\mathcal{D}_{a,b}|=p(p+1)$. A representative is given by $\left( \begin{smallmatrix} a & 0 \\ 0 & b \end{smallmatrix} \right)$.
\item Let $\mathcal{N}_a$ be the set of non-diagonal matrices with one single eigenvalue $a \in \F_p^\times$. For each of the $p-1$ choices for $a$, we have $|\mathcal{N}_a|=p^2-1$. One may choose as a representative element $aJ$ or $aJ_\epsilon$.
\item Let $\mathcal{D}_a$ denote the singleton $\left\{ \left( \begin{smallmatrix}
    a & 0 \\ 0 & a
\end{smallmatrix} \right) \right\}$ for each $a \in \F_p^\times$. There are thus $p-1$ such conjugacy classes, and each has size $1$.
\item Let $\mathcal{E}_\tau$ be the set of matrices whose eigenvalues are $\tau$ and $\tau'$, where $\tau \in \F_{p^2}\setminus\F_p$ and $\tau'$ is the conjugate of $\tau$. There are $\frac{1}{2}(p^2-p)$ choices for $\tau$, and for each choice we have $|\mathcal{E}_\tau|=p^2-p.$ 
\end{itemize}

We see that $G$ has a total of $p^2-1$ conjugacy classes. Given an element $g \in G$, let $$C_G(g)=\{h \in G \ | \ hgh^{-1}=g\}$$ denote the centralizer of $g$ in $G$. Furthermore, write
\[
\det C_G(g)=\{\det(z) \ | \ z \in  C_G(g)\} \leq \F_p^\times
\]
noting that this subgroup is well-defined up to conjugation, since replacing $g$ with a conjugate element $aga^{-1}$ would replace $C_G(g)$ with a conjugate subgroup $aC_G(g)a^{-1}$, and $\det$ is a class function.

\begin{lemma}\label{lem:conjugacy-class-centralizers}
For every $g \in G$ we have $\det C_G(g)=\F_p^\times$ \textit{unless} $g \in \mathcal{N}_a$ for some $a \in \F_p^\times$, in which case $\det C_G(g)=(\F_p^\times)^2.$
\end{lemma}
\begin{proof}
It is immediately clear that elements of $\mathcal{D}_a$ are centralized by all of $G$. A direct computation shows that elements of $\mathcal{D}_{a,b}$ are centralized by the diagonal matrices $\left( \begin{smallmatrix} c & 0\\ 0 & d \end{smallmatrix}\right)$ with $c,d \in \F_p^\times$. In both of these cases, the determinant of the centralizer is all of $\F_p^\times$.

Let $g \in \mathcal{E}_\tau$ for some $\tau \in \F_{p^2}^\times \setminus \F_p$. Thus $g$ has irreducible characteristic polynomial $\mathrm{char}(g)=x^2-\tr(g)x+\det(g)$ and is similar to the matrix $\left( \begin{smallmatrix} 0 & -\det(g)\\ 1 & \tr(g) \end{smallmatrix}\right)$. Now a direct computation shows that this matrix is centralized by matrices of the form $aI+bg$ for $a,b \in \F_p$, hence $C_G(g)\simeq \F_p[g]^\times \simeq \F_{p^2}^\times$. On the other hand, a matrix of the form $aI+bg$ has eigenvalues $a+b\tau$ and $a+b \tau'$, so 
\begin{align*} 
\det(aI+bg)&=(a+b\tau)(a+b\tau')\\
&=a^2 + ab(\tau + \tau') + b^2 \tau \tau' \\
&=a^2 + ab \tr(g) + b^2 \det(g). 
\end{align*}
But identifying $\F_{p^2} \simeq \F_p[\tau]$, we see that we may identify the determinant map with the field norm $\F_{p^2} \to \F_p$, which is surjective. So once again the determinant of the centralizer is all of $\F_p^\times$.

For the remaining case, a direct computation shows that elements of $\mathcal{N}_a$ are centralized by matrices of the form $\left( \begin{smallmatrix} b & c\\ 0 & b \end{smallmatrix}\right)$ with $b \in \F_p^\times$ and $c \in \F_p$, which have determinant $b^2$.
\end{proof}

For each $t \in \F_p^\times$, define $N_t$ to be the number of $G$-conjugacy classes with fixed determinant $t$. 

\begin{proposition}\label{prop:number-classes-fixed-determinant}
The value of $N_t$ is given by
\[
N_t = \begin{cases}
p+2, & \text{if } \chi(t)=1\\
p, & \text{if } \chi(t)=-1\\
\end{cases}.
\]
\end{proposition}
\begin{proof}

First suppose $\chi(t)=-1$, so $t \notin (\F_p^\times)^2$. From the descriptions above, we see that there are no contributions from the conjugacy classes of type $\mathcal{D}_a$ or $\mathcal{N}_a$, since those consist entirely of matrices with square determinants. On the other hand, for each $a \in \F_p^\times$, we have a contribution from $\mathcal{D}_{a,ta^{-1}}$. As noted earlier, we have $\mathcal{D}_{a,ta^{-1}}=\mathcal{D}_{ta^{-1},a}$, so we obtain a clean double-count unless we ever have $a=ta^{-1}$, or equivalently $t=a^2$. Since $\chi(t)=-1$, this never happens, so these conjugacy classes contribute $\frac{p-1}{2}$ to our count. Similarly, $\mathcal{E}_\tau$ contributes a class whenever $\tau \in \F_{p^2}\setminus \F_p$ satisfies $N(\tau)=t$, where $N \colon \F_{p^2}^\times \to \F_p^\times$ denotes the norm map. Since the norm map for this quadratic extension is surjective with fibers of size $p+1$, we get $\frac{p+1}{2}$ contributions after removing double counts.

Now suppose $\chi(t)=1$, so there exists $a \in \F_p^\times$ with $a^2=t$. From the description above, we see that $\mathcal{D}_{a}$ and $\mathcal{D}_{-a}$ each contribute a class, as do $\mathcal{N}_{a}$ and $\mathcal{N}_{-a}$. On the other hand, since $a^2=t$, a repeated eigenvalue contributes classes of type $\mathcal{D}_\bullet$ and $\mathcal{N}_\bullet$, but not of type
$\mathcal{D}_{\bullet,\bullet}$, so we now obtain only $\frac{p-3}{2}$ classes of type $\mathcal{D}_{\bullet,\bullet}$. Similarly, since two elements of the norm fiber above $t$ actually belong to $\F_p^\times$ (namely $\pm a$), we only obtain $\frac{p-1}{2}$ classes of type $\mathcal{E}_\bullet$ in this case. Thus, compared with the nonsquare case, the square case gains two
$\mathcal{D}_\bullet$ and two $\mathcal{N}_\bullet$ classes while losing one class each of type $\mathcal{D}_{\bullet,\bullet}$ and $\mathcal{E}_\bullet$, for a total of $2+2+\frac{p-3}{2}+\frac{p-1}{2}=p+2$.

\end{proof}

\subsubsection{Conjugacy classes of $\Gamma_\mathrm{id}$ and $\Gamma_\chi$}\label{subsubsec:conjugacy-classes-of-gamma-id-and-chi}


\begin{proposition}\label{prop:conjugacy-classes-of-Gamma-subgroups} We have isomorphisms
\[
\phi_\mathrm{id} \colon G \xrightarrow{\sim} \Gamma_\mathrm{id}, \quad g \mapsto (g,g)
\]
and 
\[
\phi_\chi \colon G \xrightarrow{\sim} \Gamma_\chi, \quad g \mapsto (g,\chi(\det g)g)
\]
which preserve conjugacy classes. 
\end{proposition}
\begin{proof}
That these maps are isomorphisms is immediate, so we check the claim about conjugacy classes. This is a direct computation: for $g,h \in G$ we have 
\[
(h,h)(g,g)(h,h)^{-1}=(hgh^{-1},hgh^{-1}) = \phi_\mathrm{id}(hgh^{-1}).
\]
and 
\[
(h,\chi(\det h)h)(g, \chi(\det g)g)(h,\chi(\det h)h)^{-1} = (hgh^{-1},\chi(\det hgh^{-1}) hgh^{-1}) = \phi_\chi(hgh^{-1}).
\]
This shows that the maps respect conjugacy classes, as claimed.
\end{proof}


\subsubsection{Conjugacy classes of $H_\pm$}\label{subsubsec:conjugacy-classes-of-H-pm}

\begin{proposition}\label{prop:conjugacy-classes-of-Hpm} We have an isomorphism
\[
\phi_\pm \colon H_\pm \xrightarrow{\sim} G \times \{\pm 1\}, \quad (g, \epsilon g) \mapsto (g,\epsilon)
\]
of direct products. 
\end{proposition}
\begin{proof}
This is a straightforward computation.
\end{proof}

Thus the conjugacy classes of $H_\pm$ are of the form $C \times \{1\}$ and $C \times \{-1\}$, where $C$ is a conjugacy class of $G$, for a total of $2(p^2-1)$ conjugacy classes in $H_\pm.$ This two-to-one correspondence between conjugacy classes of $H_\pm$ and $G$ preserves the size.

\subsubsection{Conjugacy classes of $\Delta$}\label{subsubsec:conjugacy-classes-of-Delta}

Let us set some additional notation. Given $g \in G$, let $$\mathcal{C}_g = \{ hgh^{-1} \ | \ h \in G\}$$ denote the orbit of $g$ under conjugation. Given $(g,h) \in \Delta$, we have a $(G \times G)-$conjugacy class $\mathcal{C}_g \times \mathcal{C}_h$, and the next theorem describes the $\Delta$-conjugacy classes in  $(\mathcal{C}_g \times \mathcal{C}_h) \cap \Delta$.

\begin{theorem}\label{thm:conjugacy-classes-of-Delta}
Let $(g,h)\in \Delta$. 
\begin{enumerate}
\item If at least one of $\mathcal{C}_g,\mathcal{C}_h$ is of type different from $\mathcal{N}_a$, then $(\C_g \times \C_h)\cap \Delta$ contains exactly one conjugacy class. Thus $(g,h)$ and $(g',h')$ are $\Delta$-conjugate if and only if $g$ is $G$-conjugate to $g'$ and $h$ is $G$-conjugate to $h'$.

\item If both $\mathcal{C}_g$ and $\mathcal{C}_h$ are of type $\mathcal{N}_\bullet$, then $(\C_g \times \C_h)\cap \Delta$ contains exactly two conjugacy classes with representatives 
\[
(\lambda J, \pm \lambda J) \quad \text{and} \quad (\lambda J, \pm \lambda J_\epsilon)
\]
for some $\lambda \in \F_p^\times.$
\end{enumerate}
\end{theorem}

\begin{proof}
Since $\det$ is a class function and $\det g=\det h$ (since $(g,h) \in \Delta$ by assumption), the full
$(G\times G)$-class $X=\mathcal{C}_g\times\mathcal{C}_h$ of $(g,h)$
is contained in $\Delta$, and is a union of $\Delta$-conjugacy classes. 

Every element of $X$ has the form $(g',h')=(aga^{-1},bhb^{-1})$, and the question is whether the conjugating pair $(a,b) \in G \times G$ can
be chosen with $\det a=\det b$. Note that conjugation by $(a,b)$ is determined up to right multiplication by an element of $C_G(g) \times C_G(h)$, since (for instance) if $z\in C_G(g)$ then
\[
(az)g(az)^{-1}=azgz^{-1}a^{-1}=agzz^{-1}a^{-1} = aga^{-1},
\]
and conversely we have 
\[
aga^{-1}=bgb^{-1} \Longrightarrow b^{-1}a \in C_G(g).
\]
Write $D_1=\det C_G(g)$ and $D_2=\det C_G(h)$, in the notation of Lemma~\ref{lem:conjugacy-class-centralizers}. Thus $(g',h')$ is a $\Delta$-conjugate of $(g,h)$ precisely when $\det(a)\det(b)^{-1} \in D_1D_2$.

Let us then consider the map 
\[
\pi \colon X \to \F_p^\times / D_1D_2, \qquad (aga^{-1},bhb^{-1}) \mapsto \det(a)\det(b)^{-1} \mod D_1 D_2.
\]
Note that we may also abuse notation and just refer to $\pi(a,b)$ in terms of the conjugating element. 

It is clear from the preceding discussion that this map is well-defined, and it is surjective since we may allow $a,b$ to range over all of $G$. Note that if $(u,v) \in \Delta$, then $$\pi(au,bv)=\det(au)\det(bv)^{-1}=\det(a)\det(b)^{-1}\det(u)\det(v^{-1})=\det(a)\det(b)^{-1}.$$ 
This shows that $\pi$ is constant on $\Delta$-orbits, and conversely, different $\Delta$-orbits correspond to different values under $\pi$. Thus the fibers of $\pi$ are precisely the $\Delta$-orbits, and the number of fibers is given by the index $[\F_p^\times \colon D_1D_2]$.  

By Lemma~\ref{lem:conjugacy-class-centralizers}, this index is $1$ in every case \textit{except} for when $D_1D_2=(\F_p^\times)^2$, which occurs precisely when $\C_g$ and $\C_h$ are each of type $\mathcal{N}_\bullet$. In this case the index is $2$, and we may describe the fibers explicitly. We may assume $g=\lambda J$ for some $\lambda \in \F_p^\times$, and then since $\det(g)=\det(h)$ and $h \in \mathcal{N}_a$ for some $a$, we have either $h=\pm \lambda J$ or $h= \pm \lambda J_\epsilon$. While $J$ and $J_\epsilon$ are $G$-conjugate, we will show that the $\pi$-fibers split along these two representatives. To that end, suppose $b \in G$ such that $bJb^{-1}=J_\epsilon$. A direct computation shows that $b$ is of the form $b=\left( \begin{smallmatrix}\alpha & \ast  \\ 0 & \alpha \epsilon^{-1}  \end{smallmatrix}\right)$, hence it has nonsquare determinant. Thus $(\lambda J, \pm \lambda J)$ is conjugated to $(\lambda J, \pm \lambda J_\epsilon)$ by the pair $(1,b)$, and $\pi(1,b)=\det(b)^{-1} \notin (\F_p^\times)^2$. This shows that $(\lambda J,\pm \lambda J)$ and $(\lambda J,\pm \lambda J_\epsilon)$ are representatives for the two distinct $\Delta$-conjugacy classes into which $\C_g \times \C_h$ splits. 
\end{proof}

\begin{remark}\label{rmk:tr-tr-det-doesnt-distinguish-delta-classes}
We note that the matrices $\pm \lambda J$ and $\pm \lambda J_\epsilon$ have the same characteristic polynomial, and so the data $(\tr(g),\tr(h),\det(h))$ is not enough to distinguish the two classes arising in that case.  
\end{remark}

\begin{proposition}\label{prop:number-of-delta-conjugacy-classes}
The total number of $\Delta$-conjugacy classes is $(p-1)(p^2+2p+4).$
\end{proposition}
\begin{proof}
Fix $t \in \F_p^\times.$ Recall our notation $N_t$ from Proposition~\ref{prop:number-classes-fixed-determinant}, which counts the number of $G$-conjugacy classes with determinant $t$. 

From the previous theorem, the $\Delta$-conjugacy classes are almost entirely made up of pairs of $G$-conjugacy classes, \textit{except} for when we have the splitting of e.g. $\mathcal{N}_\lambda \times \mathcal{N}_\lambda$ into two classes. Since every class of type $\mathcal{N}_\bullet$ has square determinant, let us define
\[
n_t=\begin{cases}
2, & \text{if}\ \chi(t)=1\\ 
0 & \text{if}\ \chi(t)=-1
\end{cases}.
\]
Then we may count the number of $\Delta$-conjugacy classes by summing over the values of $t \in \F_p^\times$, splitting up the computation according to whether $t \in (\F_p^\times)^2$. To be explicit, suppose $\C_1 \times \C_2$ is a $(G \times G)$-conjugacy class with $\det \C_1 = \det \C_2 = t$. If $t \not\in (\F_p^\times)^2$, then there are $N_t$ choices for each of $\C_1$ and $\C_2$, for a contribution of $N_t^2$. On the other hand, if $t \in (\F_p^\times)^2$, then we get an extra $n_t^2=4$ contributions from the pairs where $\C_1, \C_2$ have representatives in each of $J,J_\epsilon.$

Using Proposition~\ref{prop:number-classes-fixed-determinant}, we thus compute
\begin{align*}
\#(\Delta\text{-conjugacy classes}) &= \sum_{t \in \F_p^\times} (N_t^2 + n_t^2) \\ 
&=\sum_{t \in (\F_p^\times)^2} (N_t^2 + n_t^2) + \sum_{t \notin (\F_p^\times)^2} (N_t^2 + n_t^2) \\ 
&= \frac{p-1}{2} \left[ (p+2)^2 + 4 \right]  + \frac{p-1}{2} p^2 \\ 
&= (p-1)(p^2 + 2p + 4)
\end{align*}
as claimed.
\end{proof}

\subsection{Computing the witness ratios}\label{subsec:witness-ratios}

We will now compute the witness ratios for each admissible subgroup. First we need a lemma.

\begin{lemma}\label{lem:Ntd}
Let $$N(t,d)=\#\{g \in G \ | \ \tr(g)=t \quad \text{and} \quad \det(g)=d\}.$$ Then
\[
N(t,d)=\begin{cases}
p^2+p, &\text{if } \chi(t^2-4d)=1,\\
p^2, &\text{if } t^2-4d=0,\\
p^2-p, &\text{if } \chi(t^2-4d)=-1\\
\end{cases}.
\]
In particular, we have 
\[
\#\{g \in G \ | \ \tr(g)=0\} = p^2(p-1)
\]
and
\[
\#\{g \in G \ | \ \tr(g)=0 \text{ and } \chi(\det(g))=-1\} = \begin{cases}
    \frac{1}{2}p(p-1)^2, &\text{if } p \equiv 1 \mod 4\\
    \frac{1}{2}p(p^2-1), &\text{if } p \equiv 3 \mod 4\\
\end{cases}.
\]
\end{lemma}
\begin{proof}
The value of $N(t,d)$ is equivalent to the number of elements of $G$ with characteristic polynomial $x^2-tx+d$, which has discriminant $\delta=t^2-4d$.

If $\chi(\delta)=1$, then $g$ has two distinct eigenvalues, so the contributions to $N(t,d)$ come from a single conjugacy class of type $\mathcal{D}_{a,b}$, and each such class has $p^2+p$ elements.

If $\chi(\delta)=-1,$ then the characteristic polynomial is irreducible, so the contributions to $N(t,d)$ come from a single conjugacy class of type $\mathcal{E}_\tau$, and each such class has $p^2-p$ elements.

If $\delta=0$, then $g$ has a repeated eigenvalue. In this case, $N(t,d)$ gets a contribution of size $p^2-1$ from a class of type $\mathcal{N}_a$ and a contribution of size $1$ from a class of type $\mathcal{D}_a$, for a total of $p^2$ such matrices.

We may now compute
\begin{align*}
\#\{g \in G \ | \ \tr(g)=0\} &= \sum_{d \in \F_p^\times} N(0,d) \\ 
&= \sum_{\substack{d \in \F_p^\times\\ \chi(-d)=1}} (p^2 +p)\ + \sum_{\substack{d \in \F_p^\times\\ \chi(-d)=-1}} (p^2 -p)\\
&=\frac{p-1}{2}(p^2+p) + \frac{p-1}{2}(p^2-p)\\ 
&=p^2(p-1).
\end{align*}
The proof of the final claim is a similar straightforward computation, noting that $\chi(-1)=1$ if and only if $p \equiv 1 \mod 4.$
\end{proof}

\begin{theorem}\label{thm:witness-ratios-general-p}
Let $p \geq 5$. Then the witness ratios for the admissible subgroups are
\begin{align*}
W(\Gamma_\mathrm{id})=0, \quad W(\Gamma_\chi)=\frac{1}{2} - \frac{1}{2(p+\chi(-1))}, \quad
W(H_\pm)=\frac{p^2-p-1}{2(p^2-1)}, \quad 
W(\Delta)=1-\frac{p^3-p-1}{(p^2-1)^2}.
\end{align*}
\end{theorem}
\begin{proof}
\begin{enumerate}
\item Since $\Gamma_\mathrm{id}=\{(g,g) \ | \ g \in G\}$, it is clear that $W(\Gamma_\mathrm{id})=0$.

\item Let us now consider $\Gamma_\chi$, whose elements are of the form $(g_1,g_2)$ with $g_2=\chi(\det g_1)g_1$. Thus
\[
\tr(g_2)=\chi(\det g_1) \tr(g_1),
\]
so $\tr(g_1) \neq \tr(g_2)$ if and only if $\chi(\det g_1) = -1$ and $\tr(g_1) \neq 0$. Thus
\[
W(\Gamma_\chi) = \frac{1}{|G|} \cdot \#\{g \in G \ | \ \tr(g) \neq 0 \quad \text{and} \ \chi(\det g)=-1\}
\]
From Proposition \ref{prop:conjugacy-classes-of-Gamma-subgroups} we know that $\Gamma_\chi \simeq G$, the map $\det \colon G \to \F_p^\times$ is surjective with fibers of size $|S|$, and half of the values in $\F_p^\times$ are squares. So a full $\frac{1}{2}$ of the elements of $\Gamma_\chi$ do not contribute to our count, since they have square determinant. From those that remain, we must discard those with $\tr(g)=0$, so we have
\begin{align*}
W(\Gamma_\chi) &= \frac{1}{2} -  \frac{\#\{g \in G \ | \ \tr(g) = 0 \quad \text{and} \ \chi(\det g)=-1\}}{|G|}.
\end{align*}
The claimed formula now follows from the previous lemma and the fact that $|G|=p(p-1)^2(p+1).$
\item Let us now consider $H_\pm$, whose elements are of the form $(g,\pm g)$. As in the previous case, a full half of the elements do not contribute to our count, namely those of the form $(g,g)$. The only elements which do contribute to our count, then, are those of the form $(g,-g)$ with $\tr(g) \neq 0$. Using the previous lemma, we thus have
\[
W(H_\pm) = \frac{1}{2}\left(1-\frac{\#\{g \in G \ | \ \tr(g)=0\}}{|G|}\right) = \frac{1}{2}\left(1-\frac{p^2(p-1)}{p(p-1)^2(p+1)}\right) = \frac{p^2-p-1}{2(p^2-1)}.
\]
\item Finally, consider $\Delta$, consisting of $\textit{all}$ pairs $(g_1,g_2)$ with $\det(g_1)=\det(g_2)$. Set $d=\det(g_1)$. Then in the notation of Lemma~\ref{lem:Ntd}, there are $N(t,d)$ ways to choose $g_2$ such that $(g_1,g_2) \in \Delta$ and $\tr(g_2)=t$. The only ones which do \textit{not} contribute to the count for $W(\Delta)$ are those for which $t=\tr(g_1)$. Taking into account the symmetry between $g_1$ and $g_2$ in this argument, it follows that 
\[
1-W(\Delta) = \frac{1}{|\Delta|}\sum_{d \in \F_p^\times} \sum_{t \in \F_p} N(t,d)^2. 
\]
We will compute the inner sum for each fixed $d \in \F_p^\times,$ using Lemma~\ref{lem:Ntd}. Thus, for a fixed $t$ and $d$, we consider the characteristic polynomial $x^2-tx+d$ with discriminant $\delta=t^2-4d$. 

Suppose $d \in (\F_p^\times)^2$. Then $\delta=0$ for precisely two values of $t$. Now let's consider for how many values of $t$ we have $\chi(\delta)=1$. For each such $t$, there exists $a \in \F_p^\times$ such that $(t-a)(t+a)=4d$. Equivalently, we seek to count solutions $(u,v)$ to $uv=4d$, and each solution corresponds to $t=\frac{u+v}{2} \in \F_p$ and $a=\frac{v-u}{2}\in \F_p^\times.$ But clearly we are free to choose $u \in \F_p^\times$, and then $v=4du^{-1}$ gives a solution. However, we require $a \in \F_p^\times$, so we must not have $u=v$, or equivalently we must not have $4d=u^2$, and since $\chi(d)=1$ this excludes exactly two values of $u$. So we are left with $p-3$ choices for $u$, but there is a symmetry to account for, since $\pm a$ correspond to the same value of $t$, and so there are a total of $\frac{p-3}{2}$ values of $t$ for which $\chi(\delta)=1.$ 

Finally, every remaining $t$ will result in $\chi(\delta)=-1$. We have accounted for $2+\frac{p-3}{2}=\frac{p+1}{2}$ values in $\F_p$, leaving $\frac{p-1}{2}$ remaining.

In summary, when $\chi(d)=1$, we have
\begin{align*}
\sum_{t \in \F_p}N(t,d)^2 &= \sum_{\substack{t \\ \delta=0}}\left(p^2\right)^2 + \sum_{\substack{t \\ \chi(\delta)=1}}\left(p^2+p\right)^2 + \sum_{\substack{t \\ \chi(\delta)=-1}}\left(p^2-p\right)^2 \\
&= 2\left(p^2\right)^2 + \frac{p-3}{2}\left(p^2+p\right)^2 + \frac{p-1}{2}\left(p^2-p\right)^2 \\
&= p^5 - p^3 -2p^2.
\end{align*}
A similar argument shows that when $\chi(d)=-1$, we have
\[
\sum_{t \in \F_p} N(t,d)^2 =0 \cdot (p^2)^2 + \frac{p-1}{2}(p^2+p)^2 + \frac{p+1}{2}(p^2-p)^2 = p^5-p^3
\]
Putting this all together, we have
\begin{align*}
\sum_{d \in \F_p^\times} \sum_{t \in \F_p} N(t,d)^2 &=\frac{p-1}{2} (p^5-p^3-2p^2)  +   \frac{p-1}{2} (p^5-p^3)\\ 
&=p^2(p-1)(p^3-p-1),
\end{align*}
and so
\begin{align*}
W(\Delta) &= 1 - \frac{1}{|\Delta|}\sum_{d \in \F_p^\times} \sum_{t \in \F_p} N(t,d)^2\\ 
&=1- \frac{(p^3-p-1)}{(p^2-1)^2}.\\
\end{align*}
\end{enumerate}
\end{proof}

\begin{corollary}\label{cor:distinct-ratios}
The witness ratios are distinct for every pair of admissible subgroups. In particular, we have
\[
W(H_\pm)-W(\Gamma_\chi)=\frac{- \chi(-1)}{2(p^2-1)},
\]
while $W(\Delta)-W(H_\pm) \to \frac{1}{2}$ as $p \to \infty$.
\end{corollary}
\begin{proof}
These are straightforward calculations using Theorem~\ref{thm:witness-ratios-general-p}.
\end{proof}

Putting this all together, we obtain the following theorem.

\begin{theorem}\label{thm:main-p-at-least-5}
Let $p \geq5$ be a prime. Suppose $E_1$ and $E_2$ are elliptic curves defined over $\Q$. Suppose further that $\rhobar_{E_1,p}$ and $\rhobar_{E_2,p}$ are surjective. If $\rhobar_{E_1,p} \not\simeq \rhobar_{E_2,p}$, then the image $H$ of the product representation is conjugate to one of $\Gamma_\chi, H_\pm,$ and $\Delta.$ Furthermore, we compute
\[
W(\bullet)=\begin{cases} 
    \frac{1}{2}-\frac{1}{2(p+\chi(-1))}, & \bullet = \Gamma_\chi \\ 
    \frac{p^2-p-1}{2(p^2-1)}, & \bullet = H_\pm \\
    1-\frac{p^3-p-1}{(p^2-1)^2}, & \bullet=\Delta \\
    \end{cases}.
\]
As a consequence, we have
\[
\lim_{X \to \infty} \frac{\#\{\textrm{good primes } \ell < X \ |  \ a_\ell(E_1) \not\equiv a_\ell(E_2) \mod p \}}{\#\{\textrm{good primes } \ell < X\}} = W(H).
\]
\end{theorem}

\begin{remark}
As described in the introduction, and thanks to Corollary~\ref{cor:distinct-ratios}, the results of this paper can often be used to numerically identify the image subgroup associated to a fixed pair of elliptic curves $(E_1,E_2)$ modulo $p$. The major exception, however, is attempting to distinguish between $H_\pm$ and $\Gamma_\chi$ when $p\gg 0$, since their witness ratios only differ by $O(p^{-2})$.

To distinguish between these pairs, we can use the explicit descriptions of $\Gamma_\chi$ and $H_\pm$ from Section~\ref{subsec:classifying-admissible-subgroups}. In particular, every element of $\Gamma_\chi$ is of the form $(g,\chi(\det g)g)$, while for $(g_1,g_2) \in H_\pm$ we have $g_2=\pm g_1$, and the sign is independent of $\chi$. In terms of the data of the elliptic curves, in the $\Gamma_\chi$ case we always have 
\[
a_\ell(E_2) \equiv \left(\frac{\ell}{p}\right)a_\ell(E_1),
\]
while in the $H_\pm$ case we have
\[
a_\ell(E_2) \equiv \epsilon(\ell) a_\ell(E_1), \quad \epsilon(\ell) \in \{\pm 1\}
\]
and in particular there will exist a prime $\ell$ such that $\epsilon(\ell) \neq \left(\frac{\ell}{p}\right)$. In fact, by the Chebotarev density theorem, there must be a positive density of such primes, so in practice one should be able to distinguish between $\Gamma_\chi$ and $H_\pm$ with a small additional search.
\end{remark}

\subsection{Characterizing the elliptic curves pairs which realize each image subgroup}\label{subsec:realizations}

We now use the explicit descriptions of the admissible subgroups from Section~\ref{subsec:classifying-admissible-subgroups} to show that every admissible subgroup is realized as the image of some pair of elliptic curves over $\Q$. 
 
 First we need a lemma. Set $p^\ast = \chi(-1)p$, so $\Q(\sqrt{p^\ast})$ is the unique quadratic subfield of the cyclotomic extension $\Q(\zeta_p)$. For a squarefree integer $d$, let $\chi_d \colon G_\Q \to \{\pm 1\}$ be the quadratic character associated to $\Q(\sqrt{d})/\Q.$

 \begin{lemma}\label{lem:for-realizations}
\begin{enumerate} 
\item Let $d\neq 1$ be a nonzero squarefree integer. If $E^{(d)}$ denotes the corresponding quadratic twist of $E$, then $\rhobar_{E^{(d)},p} \simeq \chi_d \otimes \rhobar_{E,p}.$ In particular, $\rhobar_{E^{(d)},p} \not\simeq \rhobar_{E,p}.$
\item If $\rhobar_{E,p}$ is surjective, then the unique quadratic subfield of the $p$-division field $\Q(E[p])$ is $\Q(\sqrt{p^\ast})$, and $\chi_{p^\ast}=\chi \circ \det \circ \rhobar_{E,p}.$
\end{enumerate}
 \end{lemma}
 \begin{proof}
The fact that $\rhobar_{E^{(d)},p} \simeq \chi_d \otimes \rhobar_{E,p}$ is standard. If $\chi_d$ does not factor through $\rhobar_{E,p}$, then it immediately follows that $\rhobar_{E^{(d)},p} \not\simeq \rhobar_{E,p}.$ Otherwise, we must have $\chi_d=\chi_{p^\ast}$, and by Lemma~\ref{lem:det-preserving} $r_\chi \not\in\Inn(G)$, so once again we may conclude $\rhobar_{E^{(d)},p} \not\simeq \rhobar_{E,p}.$

The fact that $\Q(\sqrt{p^\ast})$ is the unique quadratic subfield of $\Q(E[p])$ is also standard, see e.g. \cite[Figure 5.3]{Ade01}. The final claim now follows from the Chebotarev density theorem and the observation that 
\[
\chi_{p^\ast}(\mathrm{Frob_\ell)}=\left( \frac{\ell}{p} \right) = \chi(\ell) = \chi(\det \rhobar_{E,p}(\mathrm{Frob_\ell)).}
\]
 \end{proof}

We may now prove the following characterization of explicit realizations of each of the admissible subgroups. See also Table~\ref{table:3-p-13} in Appendix~\ref{appendix}.

\begin{theorem}\label{thm:explicit-realizations}
Let $E_1,E_2$ be elliptic curves over $\Q$ with surjective mod $p$ Galois representations $\rhobar_{E_i,p}$, where $p \geq 5$. Let $H$ be the image of the product representation $\rhobar_{E_1,p} \times \rhobar_{E_2,p}$. Then $H$ is conjugate to:
\begin{enumerate}
\item $\Gamma_\mathrm{id}$ if and only if $\rhobar_{E_1,p} \simeq \rhobar_{E_2,p}$. 
\item $H_\pm$ if and only if $\rhobar_{E_2,p} \simeq \chi_{d} \otimes \rhobar_{E_1,p}$ for some squarefree $d \notin\{1,p^\ast\}$. Thus pairs of quadratic twists $(E,E^{(d)})$ give examples of $H_\pm$.
\item $\Gamma_\chi$ if and only if $\rhobar_{E_2,p} \simeq \chi_{p^\ast} \otimes \rhobar_{E_1,p}$.  Thus pairs of quadratic twists $(E,E^{(p^\ast)})$ give examples of $\Gamma_\chi$.
\item $\Delta$ in every other case.
\end{enumerate}
\end{theorem}
\begin{proof}
The claim about $\Gamma_\mathrm{id}$ is immediate. 

If $H$ is not conjugate to $\Gamma_\mathrm{id}$ or $\Delta$, then by Theorem~\ref{thm:class} it is conjugate to either $H_{\pm}$ or $\Gamma_\chi$, noting that $\Gamma_{\chi} \leq H_{\pm}$. In particular, $H$ is conjugate to a subgroup of $H_\pm.$ Then there exists $a,b \in G$ such that for every $\sigma \in G_{\Q}$ we have
\[
\rhobar_{E_1,p}(\sigma)=ag_\sigma a^{-1} \quad \text{and} \quad \rhobar_{E_2,p}(\sigma)=\epsilon(\sigma)bg_\sigma b^{-1}
\]
for $g_\sigma \in G$ and $\epsilon(\sigma) \in \{ \pm 1\}$. Consider the matrix $M=ba^{-1}$; then the displayed equations imply
\begin{equation}\label{eq:h-pm-conj}
\rhobar_{E_2,p}(\sigma)=\epsilon(\sigma)M\rhobar_{E_1,p}(\sigma)M^{-1}.
\end{equation}
Since scalars commute and the $\rhobar_{E_i,p}$ are homomorphisms, we see that $\epsilon \colon G_\Q \to \{\pm 1\}$ is also a homomorphism, and more specifically a quadratic character. If it were trivial, then \eqref{eq:h-pm-conj} would imply $\rhobar_{E_1,p}\simeq \rhobar_{E_2,p}$, but then we must have $H=\Gamma_\mathrm{id}$. So instead we have $\epsilon=\chi_d$ for some squarefree integer $d$. Conversely, given any such $\chi_d$, if $\rhobar_{E_2,p} \simeq \chi_d \otimes \rhobar_{E_1,p}$ then up to conjugation we have 
\[
\rhobar_{E_2,p}(\sigma) = \chi_d(\sigma) \rhobar_{E_1,p}(\sigma) \in \{ \pm \rhobar_{E_1,p}(\sigma) \} \quad \text{for all } \sigma \in G_\Q,
\]
hence $H \subset H_\pm.$

It thus remains to distinguish between the cases in (2) and (3) of the theorem statement. To that end, define a homomorphism $\phi \colon G_\Q \to H_\pm \simeq G \times \{\pm 1\}$ by $\phi(\sigma)=(\rhobar_{E_1,p}(\sigma),\chi_d(\sigma)),$ and consider its image. By Goursat's lemma, this corresponds to a choice of normal subgroups $N_1 \triangleleft G$ and $N_2 \triangleleft \{\pm 1\}$ such that $N_1 \times N_2 \leq H$, along with an isomorphism $G/N_1 \xrightarrow{\sim} \{\pm 1 \} /N_2.$ But $G$ has a unique index $2$ subgroup, namely $S$, and the unique surjection $G \to \{\pm 1\}$ is $\chi \circ \det$, so by Lemma~\ref{lem:for-realizations}, if $N_1=S$ then $\chi_d=\chi_{p^\ast}$. In this case we have $\rhobar_{E_2,p}(\sigma) = \chi(\det \rhobar_{E_1,p}(\sigma))\rhobar_{E_1,p}(\sigma)$ for all $\sigma \in G_\Q$, hence $H$ is conjugate to $\Gamma_\chi$. Otherwise, we must have $N_1=G$, $\chi_d$ is a nontrivial quadratic character for $d \neq p^{\ast}$, and $H$ is conjugate to $H_\pm$.
\end{proof}

\begin{remark}
The group $\Gamma_\mathrm{id}$ is realized by any pair $(E_1,E_2)$ possessing a degree prime-to-$p$ isogeny. The Frey-Mazur Conjecture asserts that these are the only pairs realizing $\Gamma_\mathrm{id}$ when $p \gg 0$. On the other hand, there are parameterized families of non-isogenous primes for small primes; see e.g. \cite{rubin-silverberg}.
\end{remark}

\section{The anomalous case $p=3$}\label{sec:main-results}

Prior to obtaining our general results in Section~\ref{sec:general-p}, we performed explicit computations for primes $p \leq 17$ using the code found in \cite{code}.  The primary scripts are implemented in Magma \cite{Magma} and SageMath \cite{sage}; the SageMath implementation also makes use of PARI \cite{PARI2} and GAP \cite{GAP4}.  While these computational results are subsumed by the results in Section~\ref{sec:general-p} for all $p \geq 5$, the case when $p=3$ is slightly anomalous, so we treat it individually in this section. 

In particular, Proposition~\ref{prop:normal-subs-are-in-center-or-big} fails when $p=3$, as noted in Remark~\ref{rem:prop-fails-p3}. In this case, there is one additional admissible subgroup, furnished via Goursat's lemma by the exceptional normal subgroup $Q_8 \leq \GL_2(\F_3)$, which is the unique normal subgroup of order 8. We call this admissible subgroup $H_Q$. Every other admissible subgroup is described by the results in Section~\ref{sec:general-p}; that is, the sole difference between the results for $p=3$ and $p \geq 5$ is that $H_Q$ exists if and only if $p=3.$ 

 We apply the methods described in Section~\ref{sec:background} for $p = 3$.  Let $E_1$ and $E_2$ be elliptic curves defined over $\Q$. Further, suppose that $\rhobar_{E_1,3}$ and $\rhobar_{E_2,3}$ are surjective. If $\rhobar_{E_1,3} \not\simeq \rhobar_{E_2,3}$, then the image of the product representation is conjugate to one of $\Gamma_\chi, H_\pm, H_Q,$ and $\Delta.$ All of these groups are described explicitly in Section~\ref{subsec:classifying-admissible-subgroups} except for $H_Q,$ which has size $2^7 \cdot 3$ and is generated by the elements
 \[
 (A, A),\quad (B, B^\top)
 \]
 where
 \begin{math}
     A =
     \left(
     \begin{smallmatrix}
         1 & 2\\
         0 & 1
     \end{smallmatrix}
     \right)
 \end{math}
 and
  \begin{math}
     B =
     \left(
     \begin{smallmatrix}
         1 & 1\\
         2 & 1
     \end{smallmatrix}
     \right)
 \end{math}.
 
Using our code from \cite{code}, we computed the witness ratios for each of the admissible subgroups in the $p=3$ case, giving the following theorem. (We note that, for the subgroups $\Gamma_\chi, H_\pm$, and $\Delta$, our computations yield the same witness ratios as the general formulas from Section~\ref{sec:general-p}, so once again, the only new information here pertains to $H_Q$.)

\begin{theorem}\label{thm:main-p-3}
Suppose $E_1$ and $E_2$ are elliptic curves defined over $\Q$. Suppose further that $\rhobar_{E_1,3}$ and $\rhobar_{E_2,3}$ are surjective. If $\rhobar_{E_1,3} \not\simeq \rhobar_{E_2,3}$, then the image $H$ of the product representation is conjugate to one of $\Gamma_\chi, H_\pm, H_Q,$ and $\Delta.$ Furthermore, we compute
\[
W(\bullet)=\begin{cases} 
    \frac{1}{4}, & \bullet = \Gamma_\chi \\ 
    \frac{5}{16}, & \bullet = H_\pm \\
    \frac{35}{64}, & \bullet = H_Q \\
    \frac{41}{64}, & \bullet=\Delta \\
    \end{cases}.
\]
As a consequence, we have
\[
\lim_{X \to \infty} \frac{\#\{\textrm{good primes } \ell < X \ |  \ a_\ell(E_1) \not\equiv a_\ell(E_2) \mod 3 \}}{\#\{\textrm{good primes } \ell < X\}} = W(H).
\]
\end{theorem}

\begin{proof}
As described before the theorem, we originally used a computer algebra system to determine the possible groups which can arise as $\mathrm{Im}(\rhobar)$. With these in hand, for each such group $H$ we can again use our scripts \cite{code} to enumerate the conjugacy classes, making note of a representative element and the size for each class. It is then a straightforward matter to compute $W(H)$. The final statement follows immediately from the Chebotarev Density Theorem. 
\end{proof} 


For the remainder of this section, we consider the exceptional admissible group $H_Q$. As confirmed by our explicit computations, the extra subgroup $Q_8$ gives rise to precisely one more admissible group, since $G/Q_8 \simeq S_3$ and $\Aut(S_3)=\Inn(S_3)$. Thus, by Goursat's lemma we must have the description
\[
H_Q = \{ (g_1,g_2) \in \GL_2(\F_3) \times \GL_2(\F_3) \ | \ g_1 g_2^{-1} \in Q_8\}.
\]
Since $\GL_2(\F_3)$ has order $48$ and $Q_8$ is its unique normal subgroup of order $8$, we have $|\GL_2(\F_3)/Q_8|=6$. One checks that $G/Q_8$ is nonabelian, hence $\GL_2(\F_3) / Q_8 \simeq S_3$. For $i \in \{1,2\}$ let us write $K_{E_i}=\Q(E_i[3])^{Q_8}$ for the subfield of the $3$-division field of $E_i$ fixed by $Q_8$. By the Galois correspondence, this is the unique $S_3$-extension of $\Q$ inside of $\Q(E_i[3])$. In light of Theorem~\ref{thm:explicit-realizations}, we see that the image of the product representation is conjugate to $H_Q$ precisely when $K_{E_1}=K_{E_2}$ but $\rhobar_{E_2} \not\simeq \chi_d \otimes \rhobar_{E_1}$ for any character $\chi_d$ of order 1 or 2. 

For an elliptic curve $E$, let $\Delta_E$ denote its discriminant. The criterion derived above can be translated into the following theorem in terms of the discriminants of the elliptic curves.

\begin{theorem} \label{thm:3-delta} Let $E_1$ and $E_2$ be elliptic curves with surjective mod $3$ Galois representations. The pair $(E_1, E_2)$ realizes $H_Q$ if and only if their Galois representations are not (trivial or quadratic) twists of one another and $\Delta_{E_1} \Delta_{E_2}^{\pm1} \in (\Q^\times)^3$.
\end{theorem}

\begin{proof}
    Let $E \in \{E_1,E_2\}$ and let $K_E \subseteq \Q(E[3])$ denote the fixed
    field of $Q_8$. Since $Q_8$ is the unique normal subgroup of $\GL_2(\F_3)$ of
    index $6$ \cite[Figure 5.1]{Ade01}, and $\rhobar_{E,3}$ is surjective, $K_E$ is the unique Galois subextension of $\Q(E[3])/\Q$ of degree $6$. By \cite[Proposition 5.4.3]{Ade01}, $\Q(\zeta_3,\Delta_E^{1/3})$ is precisely the fixed field of $Q_8$; indeed, the cubic resolvent of the defining polynomial of $\Q(x(E[3]))/\Q$ is $x^3-\Delta_E$ after normalization. In particular we have $[\Q(\zeta_3,\Delta_E^{1/3}):\Q]=6,$ so $\Delta_E$ is not a rational cube, and $K_E=\Q(\zeta_3,\Delta_E^{1/3}).$
    
    By Kummer theory,
    $\Q(\zeta_3,\Delta_{E_1}^{1/3})=\Q(\zeta_3,\Delta_{E_2}^{1/3})$ if and only if
    $\Delta_{E_1}\Delta_{E_2}^{\pm 1}\in(\Q^\times)^3$, that is, if and only if
    $\Delta_{E_1}$ and $\Delta_{E_2}$ generate the same subgroup of
    $\Q^\times/(\Q^\times)^3$.

    It remains to translate this into a statement about images. The image of
    $\rhobar_{E_1}\times\rhobar_{E_2}$ is contained in $H_Q$ if and only if $\rhobar_{E_1}$ and
    $\rhobar_{E_2}$ agree modulo $Q_8$, which by the Galois correspondence holds if and
    only if $K_{E_1}=K_{E_2}$. Since both projections are surjective, such an image
    is conjugate to $H_Q$ unless it is one of the finer images
    $\Gamma_{\mathrm{id}}$, $\Gamma_\chi$, or $H_\pm$, and by Theorem~\ref{thm:explicit-realizations}
    these occur precisely when $\rhobar_2\cong\chi_d\otimes\rhobar_1$ for some
    quadratic (possibly trivial) character $\chi_d$. Excluding this by hypothesis,
    the image is conjugate to $H_Q$ if and only if
    $\Delta_{E_1}\Delta_{E_2}^{\pm 1}\in(\Q^\times)^3$.
\end{proof}

\begin{remark}
    One can check via SageMath or Magma that the example pair $(\href{https://www.lmfdb.org/EllipticCurve/Q/11a2/}{\texttt{11a2}}, \href{https://www.lmfdb.org/EllipticCurve/Q/352d1/}{\texttt{352d1}})$ from Table~\ref{table:3-p-13} satisfies the criteria of Theorem~\ref{thm:3-delta}.
\end{remark}

\appendix

\section{Some explicit examples}\label{appendix}
\begin{table}[h!]
\caption{Subgroups, witness ratios, and explicit examples for $3 \leq p \leq 13$}\label{table:3-p-13}
\begin{tabular}{|c|c|c|c|c|}
\hline
    Prime $p$ & Subgroup & Order & Witness ratios & Explicit example\\
    \hline \hline
    3 & $\Gamma_\chi$ & $2^4 \cdot 3$ & $1/4$ & $(\href{https://www.lmfdb.org/EllipticCurve/Q/11a1/}{\texttt{11a1}}, \href{https://www.lmfdb.org/EllipticCurve/Q/99d2/}{\texttt{99d2}})$\\
     & $H_\pm$ & $2^5 \cdot 3$ & $5/16$ & $(\href{https://www.lmfdb.org/EllipticCurve/Q/11a1/}{\texttt{11a1}}, \href{https://www.lmfdb.org/EllipticCurve/Q/176b2/}{\texttt{176b2}})$ \\
     & $H_Q $ & $2^7 \cdot 3$ & $35/64$ & $(\href{https://www.lmfdb.org/EllipticCurve/Q/11a2/}{\texttt{11a2}}, \href{https://www.lmfdb.org/EllipticCurve/Q/352d1/}{\texttt{352d1}})$ \\
     & $\Delta$ & $2^7 \cdot 3^2$ & $41/64$ & $(\href{https://www.lmfdb.org/EllipticCurve/Q/11a1/}{\texttt{11a1}}, \href{https://www.lmfdb.org/EllipticCurve/Q/17a1/}{\texttt{17a1}})$\\
     \hline
     5 & $\Gamma_\chi$ & $2^5 \cdot 3 \cdot 5$ & $5/12$ & $(\href{https://www.lmfdb.org/EllipticCurve/Q/26b2}{\texttt{26b2}}, \href{https://www.lmfdb.org/EllipticCurve/Q/650f2}{\texttt{650f2}})$\\
     & $H_{\pm}$ & $2^6 \cdot 3 \cdot 5$ & $19/48$ & $(\href{https://www.lmfdb.org/EllipticCurve/Q/26b2}{\texttt{26b2}},\href{https://www.lmfdb.org/EllipticCurve/Q/1872n2}{\texttt{1872n2}})$  \\
     & $\Delta$ & $2^8 \cdot 3^2 \cdot 5^2$ & $457/576$ & $(\href{https://www.lmfdb.org/EllipticCurve/Q/14a3}{\texttt{14a3}},\href{https://www.lmfdb.org/EllipticCurve/Q/15a3}{\texttt{15a3}})$ \\
     \hline 
     7 & $\Gamma_\chi$ & $2^5 \cdot 3^2 \cdot 7$ & $5/12$ & $(\href{https://www.lmfdb.org/EllipticCurve/Q/11a2}{\texttt{11a2}},\href{https://www.lmfdb.org/EllipticCurve/Q/539d3}{\texttt{539d3}})$ \\
     & $H_{\pm}$ & $2^6 \cdot 3^2 \cdot 7$ & $41/96$ & $(\href{https://www.lmfdb.org/EllipticCurve/Q/11a1}{\texttt{11a1}},\href{https://www.lmfdb.org/EllipticCurve/Q/704a3}{\texttt{704a3}})$ \\
     & $\Delta$ & $2^9 \cdot 3^3 \cdot 7^2$ & $1969/2304$ & $(\href{https://www.lmfdb.org/EllipticCurve/Q/11a1}{\texttt{11a1}},\href{https://www.lmfdb.org/EllipticCurve/Q/19a1}{\texttt{19a1}})$\\
     \hline 
     11 & $\Gamma_\chi$ & $2^4 \cdot 3 \cdot 5^2 \cdot 11$ & $9/20$ & (\href{https://www.lmfdb.org/EllipticCurve/Q/11a2}{\texttt{11a2}}, \href{https://www.lmfdb.org/EllipticCurve/Q/121d3}{\texttt{121d3}}) \\
     & $H_\pm$ & $2^5 \cdot 3 \cdot 5^2 \cdot  11$ & $109/240$ & (\href{https://www.lmfdb.org/EllipticCurve/Q/20a4}{\texttt{20a4}}, \href{https://www.lmfdb.org/EllipticCurve/Q/100a3}{\texttt{100a3}}) \\
     & $\Delta$ & $2^7 \cdot 3^2 \cdot 5^3 \cdot 11^2$ & $13081/14400$ & $(\href{https://www.lmfdb.org/EllipticCurve/Q/11a1}{\texttt{11a1}},\href{https://www.lmfdb.org/EllipticCurve/Q/19a1}{\texttt{19a1}})$
     \\
     \hline 
     13& $\Gamma_\chi$ & $2^5 \cdot 3^2 \cdot 7 \cdot 13$ & $13/28$ & $(\href{https://www.lmfdb.org/EllipticCurve/Q/11a1}{\texttt{11a1}},\href{https://www.lmfdb.org/EllipticCurve/Q/1859a2}{\texttt{1859a2}})$\\ 
     & $H_\pm$ & $2^6 \cdot 3^2 \cdot 7 \cdot 13$ & $155/336$ & $(\href{https://www.lmfdb.org/EllipticCurve/Q/11a1}{\texttt{11a1}},\href{https://www.lmfdb.org/EllipticCurve/Q/704a2}{\texttt{704a2}})$ \\
     & $\Delta$ & $2^8 \cdot 3^3 \cdot 7^2 \cdot 13^2$ & $26041/28224$ & $(\href{https://www.lmfdb.org/EllipticCurve/Q/11a1}{\texttt{11a1}},\href{https://www.lmfdb.org/EllipticCurve/Q/19a1}{\texttt{19a1}})$ \\
     \hline 
\end{tabular}
\end{table}

\bibliographystyle{amsalpha}
\bibliography{references}

\end{document}